\documentclass[11pt,reqno]{amsart}

\usepackage[T1]{fontenc}
\usepackage[utf8]{inputenc}
\usepackage[a4paper,margin=29mm]{geometry}
\usepackage{lmodern}
\usepackage{microtype}
\microtypesetup{expansion=false}
\usepackage{amsmath,amssymb,amsthm}
\usepackage{mathrsfs}
\usepackage{enumitem}
\usepackage{xcolor}
\usepackage{aliascnt}
\usepackage{hyperref}
\usepackage[nameinlink,capitalize,noabbrev]{cleveref}

\hypersetup{
  pdftitle={E1-Degeneration for Irregular Hodge Filtrations on Deligne--Mumford Stacks},
  pdfauthor={Haoxu Wang},
  colorlinks=true,
  linkcolor=blue!55!black,
  citecolor=green!40!black,
  urlcolor=blue!65!black
}

\numberwithin{equation}{section}

\newtheorem{theorem}{Theorem}[section]

\newaliascnt{proposition}{theorem}
\newtheorem{proposition}[proposition]{Proposition}
\aliascntresetthe{proposition}

\newaliascnt{lemma}{theorem}
\newtheorem{lemma}[lemma]{Lemma}
\aliascntresetthe{lemma}

\newaliascnt{corollary}{theorem}
\newtheorem{corollary}[corollary]{Corollary}
\aliascntresetthe{corollary}

\theoremstyle{definition}
\newaliascnt{definition}{theorem}
\newtheorem{definition}[definition]{Definition}
\aliascntresetthe{definition}

\theoremstyle{remark}
\newaliascnt{remark}{theorem}

\aliascntresetthe{remark}

\crefname{proposition}{Proposition}{Propositions}
\Crefname{proposition}{Proposition}{Propositions}
\crefname{lemma}{Lemma}{Lemmas}
\Crefname{lemma}{Lemma}{Lemmas}
\crefname{corollary}{Corollary}{Corollaries}
\Crefname{corollary}{Corollary}{Corollaries}
\crefname{definition}{Definition}{Definitions}
\Crefname{definition}{Definition}{Definitions}
\crefname{remark}{Remark}{Remarks}
\Crefname{remark}{Remark}{Remarks}

\DeclareMathOperator{\Tr}{Tr}

\newcommand{\A}{\mathbb A}
\newcommand{\C}{\mathbb C}
\newcommand{\N}{\mathbb N}
\newcommand{\Pj}{\mathbb P}
\newcommand{\Q}{\mathbb Q}
\newcommand{\Z}{\mathbb Z}

\newcommand{\cK}{\mathcal K}
\newcommand{\cL}{\mathcal L}
\newcommand{\cO}{\mathcal O}
\newcommand{\bH}{\mathbb H}

\newcommand{\FYu}{F_{\mathrm{Yu}}}
\newcommand{\red}{\mathrm{red}}
\newcommand{\id}{\mathrm{id}}

\newcommand{\Companion}{\cite{WangCompactification}}
 
\title[Irregular Hodge Filtrations on Deligne--Mumford Stacks]
{\(E_1\)-Degeneration for Irregular Hodge Filtrations\\
on Deligne--Mumford Stacks}
\author{Haoxu Wang}
\address{Morningside Center of Mathematics, Academy of Mathematics and
Systems Science, Chinese Academy of Sciences, Beijing 100190, China}
\email{krassotkinkolya@gmail.com}
\date{}

\begin{document}

\begin{abstract}
Let \((\mathscr U,w)\) be a smooth separated Deligne--Mumford stack
over \(\C\).  Assume that its coarse space is quasi-projective.  We
prove \(E_1\)-degeneration at every rational index for its irregular
Hodge filtration.  The result holds on every good stack
compactification and every NC rational stack compactification.  In
particular, the filtration and its graded dimensions can be computed
on any such compactification.  The proof uses finite flat descent for
Kontsevich complexes.  It reduces the projective stack case to the
theorem of Esnault--Sabbah--Yu for smooth projective varieties.
\end{abstract}

\maketitle
\setcounter{tocdepth}{1}
\tableofcontents

\section{Introduction}
\label{sec:introduction}

\subsection{Background and main results}

Let \(\mathscr U\) be a smooth separated Deligne--Mumford stack of
finite type over \(\C\), and let
\[
 w:\mathscr U\longrightarrow\A^1
\]
be a regular function.  Its twisted de Rham cohomology is
\[
 H_{\mathrm{dR}}^k(\mathscr U,w)
 :=
 \bH^k\!\left(
  \mathscr U,(\Omega_{\mathscr U}^\bullet,d+dw\wedge)
 \right).
\]
For smooth varieties, Yu defined an irregular Hodge filtration
\cite{Yu}, and Esnault, Sabbah, and Yu proved \(E_1\)-degeneration
\cite{ESY}.  Chen and Yu treated nondegenerate rational
compactifications \cite{ChenYu}.  Harder and Lee use the sectorwise
stack version in mirror symmetry \cite{HarderLee}.

The companion paper \Companion{} extends the compactification theory
to Deligne--Mumford stacks.  It compares the filtered complexes from
different compactifications.  It also proves that the induced
filtration is independent of the compactification.  These results do
not imply \(E_1\)-degeneration.  The purpose of this paper is to prove
this remaining statement.

We first consider a projective good stack compactification
\((\mathscr X,\boldsymbol{\mathscr D},\bar w)\).  Thus \(\mathscr X\)
is smooth and proper, its coarse space is projective,
\(\boldsymbol{\mathscr D}=\{\mathscr D_i\}_{i\in I}\) is a labelled
SNC boundary, and \(\bar w:\mathscr X\to\Pj^1\) extends \(w\).  Put
\[
 \mathscr P=\bar w^*(\infty),
 \qquad
 \mathcal L_{\mathscr X}^a(c)
 =\Omega_{\mathscr X}^a(\log\mathscr D)
   \bigl(\lfloor c\mathscr P\rfloor\bigr),
 \quad c\in\Q.
\]
For \(\alpha\in\Q\cap[0,1)\), let
\(\Omega_{\mathscr X,\bar w}^a(\alpha)\) be the Kontsevich lattice
defined from these rounded lattices.  Put
\[
 K_{\mathscr X,\bar w}(\alpha)
 =
 \bigl(
  \Omega_{\mathscr X,\bar w}^\bullet(\alpha),
  d+d\bar w\wedge
 \bigr).
\]
We use its decreasing stupid filtration.

\begin{theorem}[Projective stack degeneration]
\label{thm:intro-projective}
For every projective good stack compactification and every
\(\alpha\in\Q\cap[0,1)\), the spectral sequence
\begin{equation}
 E_1^{p,q}
 =
 H^q\!\left(
  \mathscr X,\Omega_{\mathscr X,\bar w}^p(\alpha)
 \right)
 \Longrightarrow
 \bH^{p+q}\!\left(
  \mathscr X,K_{\mathscr X,\bar w}(\alpha)
 \right)
 \label{eq:intro-kontsevich-ss}
\end{equation}
degenerates at \(E_1\).
\end{theorem}

The companion comparison with a projective good model gives the
following rational form.

\begin{theorem}[Rational degeneration]
\label{thm:intro-rational}
Assume that the coarse moduli space of \(\mathscr U\) is
quasi-projective.  Let
\((\mathscr X,\mathscr D,\bar w)\) be an NC rational stack
compactification.  This means the following \textup{(\cite[Definition~2.3]{WangCompactification})}:
\begin{enumerate}
\item
The stack \(\mathscr X\)
is smooth and proper,

\item
\(\mathscr D=(\mathscr X\setminus\mathscr U)_{\red}\) is NC, 

\item and \(\bar w:\mathscr X\dashrightarrow\Pj^1\) restricts to \(w\) on
\(\mathscr U\).  

\item On every component where \(\bar w\not\equiv0\), write
\(\operatorname{div}(\bar w)=\mathscr Z-\mathscr P\), where
\(\mathscr Z\) and \(\mathscr P\) are effective and have no common prime
component; on a zero component, put \(\mathscr Z=\mathscr P=0\).  Then
\(|\mathscr P|\subseteq|\mathscr D|\).  

\item Moreover, on a neighbourhood
\(V\) of \(|\mathscr P|\), the divisor \(\mathscr Z|_V\) is empty or
smooth, and \(\mathscr Z|_V+\mathscr D|_V\) is reduced NC.  
\end{enumerate}
For every
\(\alpha\in\Q\cap[0,1)\), its fixed-\(\alpha\) Kontsevich spectral
sequence degenerates at \(E_1\).  Its Yu integer-slice spectral
sequences also degenerate at \(E_1\).
\end{theorem}

The first theorem is the main new result.  The second follows from it
and the companion comparison.  All indices in this paper are rational.

\subsection{The proof}

An arbitrary scheme atlas of a proper stack need not be proper.  We
instead construct a finite flat surjective lci map
\[
 q:Y\longrightarrow\mathscr X
\]
from a smooth projective scheme.  A relative Bertini refinement of
Kresch--Vistoli \cite{KreschVistoli} also makes
\(D_Y=(q^{-1}\mathscr D)_{\red}\) SNC.

Put \(P_Y=q^*\mathscr P\).  The main technical result is that, for
every \(a\geq0\) and \(c\in\Q\), pullback and trace preserve the
rounded logarithmic lattices:
\begin{align*}
 q^*&:
 \mathcal L_{\mathscr X}^a(c)
 \longrightarrow
 q_*\mathcal L_Y^a(c),\\
 \Tr_q^a&:
 q_*\mathcal L_Y^a(c)
 \longrightarrow
 \mathcal L_{\mathscr X}^a(c).
\end{align*}
For pullback, the key inequality is
\(e\lfloor cm\rfloor\leq\lfloor cem\rfloor\).  For trace, we work at
codimension-one points.  Strict henselization reduces the claim to a
direct calculation for a tame Kummer extension of DVRs.

Pullback and trace then restrict to the Kontsevich complexes, and
\(\Tr_q\circ q^*=\deg(q)\,\id\).  Thus normalized trace makes the
filtered complex on \(\mathscr X\) a retract of that on \(Y\).  The
theorem of Esnault--Sabbah--Yu gives \(E_1\)-degeneration on \(Y\).
Degeneration passes to the retract.  This proves
\cref{thm:intro-projective}.  The companion comparison then gives
\cref{thm:intro-rational}.
 \section{Preliminaries and input theorems}
\label{sec:setup}

All stacks are algebraic stacks of finite type over \(\C\).  Every
Deligne--Mumford stack is separated.  Cohomology is taken on the
lisse--\'etale site.  A finite map to a stack is always
representable.  All filtration indices are rational.

\subsection{Boundary data and compactifications}

\begin{definition}[Labelled SNC divisor]
\label{def:labelled-snc}
Let \(\mathscr X\) be smooth.  A labelled SNC divisor is a finite
family
\[
 \boldsymbol{\mathscr D}=\{\mathscr D_i\}_{i\in I}
\]
of distinct smooth effective Cartier divisors such that, for every
\(J\subseteq I\),
\[
 \mathscr D_J:=\bigcap_{j\in J}\mathscr D_j
\]
is empty or smooth of pure codimension \(|J|\).  We set
\(\mathscr D_\varnothing=\mathscr X\) and
\(\mathscr D=\bigcup_i\mathscr D_i\).
\end{definition}

This condition may be checked after a surjective \(\acute{e}\)tale
base change.

\begin{definition}[Good compactification]
\label{def:good-compactification}
A good stack compactification of \((\mathscr U,w)\) consists of a
smooth proper Deligne--Mumford stack \(\mathscr X\), an open immersion
\(\mathscr U\hookrightarrow\mathscr X\), a labelled SNC presentation
of \((\mathscr X\setminus\mathscr U)_{\red}\), and a morphism
\[
 \bar w:\mathscr X\longrightarrow\Pj^1
\]
that extends \(w\).  It is \emph{projective good} if the coarse moduli
space of \(\mathscr X\) is projective.
\end{definition}

\subsection{The Kontsevich and Yu complexes}

Write
\begin{equation}
 \mathscr P:=\bar w^*(\infty)
 =\sum_{i\in I}m_i\mathscr D_i,
 \qquad m_i\in\N.
 \label{eq:polar-divisor}
\end{equation}
Some \(m_i\) may be zero.

For \(c\in\Q\), set
\[
 \lfloor c\mathscr P\rfloor
 :=\sum_i\lfloor cm_i\rfloor\mathscr D_i.
\]
For \(a\geq0\), set
\begin{equation}
 \cL_{\mathscr X}^a(c)
 :=
 \Omega_{\mathscr X}^a(\log\mathscr D)
 \bigl(\lfloor c\mathscr P\rfloor\bigr).
 \label{eq:rounded-log-lattice}
\end{equation}
These sheaves are locally free and commute with \(\acute{e}\)tale
base change.  The exterior derivative preserves them:
\[
 d\cL_{\mathscr X}^a(c)
 \subseteq
 \cL_{\mathscr X}^{a+1}(c).
\]

Fix \(\alpha\in\Q\cap[0,1)\).  The Kontsevich lattice is
\begin{equation}
 \Omega_{\mathscr X,\bar w}^a(\alpha)
 :=
 \ker\!\left\{
  \cL_{\mathscr X}^a(\alpha)
  \xrightarrow{\ \nabla_{\bar w}\ }
  \frac{\Omega_{\mathscr X}^{a+1}(*\mathscr D)}
       {\cL_{\mathscr X}^{a+1}(\alpha)}
 \right\},
 \qquad
 \nabla_{\bar w}=d+d\bar w\wedge.
 \label{eq:kontsevich-lattice}
\end{equation}
By \cite[Proposition~3.5(i)]{WangCompactification}, the sheaf
\(\Omega_{\mathscr X,\bar w}^a(\alpha)\) is locally free of rank
\(\binom{\dim\mathscr X}{a}\), for every \(a\) and \(\alpha\in\Q\cap[0,1)\).  
These lattices form the bounded complex
\begin{equation}
 K_{\mathscr X,\bar w}(\alpha)
 :=
 \bigl(
  \Omega_{\mathscr X,\bar w}^\bullet(\alpha),
  \nabla_{\bar w}
 \bigr).
 \label{eq:kontsevich-complex}
\end{equation}
We use the decreasing stupid filtration \(\sigma_{\geq\bullet}\).
Its spectral sequence starts with
\begin{equation}
 E_1^{p,q}
 =H^q\!\left(
  \mathscr X,\Omega_{\mathscr X,\bar w}^p(\alpha)
 \right).
 \label{eq:kontsevich-e1}
\end{equation}

Put
\[
 \cK_{\mathscr X,\bar w}^\bullet
 :=
 \bigl(
  \Omega_{\mathscr X}^\bullet(*\mathscr D),
  \nabla_{\bar w}
 \bigr).
\]
For \(\lambda\in\Q\), the Yu filtration is
\begin{equation}
 \FYu^\lambda\cK_{\mathscr X,\bar w}^a
 =
 \begin{cases}
  0,&a<\lceil\lambda\rceil,\\[2pt]
  \Omega_{\mathscr X}^a(\log\mathscr D)
  \bigl(\lfloor(a-\lambda)\mathscr P\rfloor\bigr),
  &a\geq\lceil\lambda\rceil.
 \end{cases}
 \label{eq:yu-complex}
\end{equation}

\begin{theorem}[Kontsevich--Yu comparison,{\cite[Proposition~3.5]{WangCompactification}}]
\label{thm:kontsevich-yu}
For \(p\in\Z\) and \(\alpha\in\Q\cap[0,1)\), inclusion in the
meromorphic complex gives compatible quasi-isomorphisms
\begin{equation}
 \sigma_{\geq p}K_{\mathscr X,\bar w}(\alpha)
 \xrightarrow{\ \sim\ }
 \FYu^{p-\alpha}\cK_{\mathscr X,\bar w}^\bullet
 \label{eq:kontsevich-yu-comparison}
\end{equation}
and
\begin{equation}
 K_{\mathscr X,\bar w}(\alpha)
 \xrightarrow{\ \sim\ }
 \cK_{\mathscr X,\bar w}^\bullet.
 \label{eq:untruncated-comparison}
\end{equation}
For fixed \(\alpha\), the spectral sequence of
\((K(\alpha),\sigma_{\geq\bullet})\) therefore agrees from \(E_1\)
with the Yu spectral sequence for
\(F_\alpha^p:=\FYu^{p-\alpha}\).
\end{theorem}

\begin{theorem}[Projective good model]
\label{thm:projective-anchor}
If the coarse moduli space of \(\mathscr U\) is quasi-projective, then
\((\mathscr U,w)\) has a projective good stack compactification
\textup{(\cite[Proposition~5.4(ii)]{WangCompactification})}.
\end{theorem}

\begin{theorem}[Comparison of compactifications]
\label{thm:compactification-comparison}
Let \(\mathscr X_1\) and \(\mathscr X_2\) be two 
NC rational stack compactifications of the same \((\mathscr U,w)\).
For every \(\alpha\in\Q\cap[0,1)\), there are compatible
isomorphisms between the derived global sections of all levels and all
successive quotients of their fixed-\(\alpha\) Kontsevich
filtrations.  These isomorphisms identify the two spectral sequences
from \(E_1\).  They are compatible with the maps to
\(H_{\mathrm{dR}}^\bullet(\mathscr U,w)\).  The corresponding Yu
filtered diagrams and spectral sequences are also identified.  
\textup{(\cite[Proposition~3.5 and
Theorem~5.6]{WangCompactification})}
\end{theorem}

\subsection{The scheme theorem}

We also use the following scheme theorem.

\begin{theorem}[Esnault--Sabbah--Yu, {\cite[Theorem~1.2.2 and Corollary~1.4.8]{ESY}}]
\label{thm:esy}
Let \(Y\) be a smooth projective complex scheme.  Let \(D\) be an SNC
divisor, and let \(g:Y\to\Pj^1\) satisfy
\(g(Y\setminus D)\subseteq\A^1\).  Then, for every
\(\alpha\in\Q\cap[0,1)\), the spectral sequence
\[
 H^q\!\left(Y,\Omega_{Y,g}^p(\alpha)\right)
 \Longrightarrow
 \bH^{p+q}\!\left(Y,K_{Y,g}(\alpha)\right)
\]
degenerates at \(E_1\).
\end{theorem}
 \section{Finite covers compatible with the boundary}
\label{sec:covers}

\begin{lemma}[Relative Bertini with avoidance]
\label{lem:relative-bertini}
Let
\[
 f:U\longrightarrow V
\]
be a proper surjective morphism of quasi-projective schemes over
\(\C\).  Assume that every geometric fibre of \(f\) is nonempty and
pure of dimension \(r\), where \(r\geq 1\).

Let \(B\subseteq U\) be a closed subscheme such that
\[
 \dim B<r.
\]
Let
\[
 U_1,\ldots,U_s\subseteq U
\]
be smooth locally closed subschemes.

Fix a locally closed immersion
\[
 U\hookrightarrow \Pj^N.
\]
For every sufficiently large integer \(d\), there is a dense open
subset
\[
 \mathcal H_d^\circ
 \subseteq
 \mathcal H_d
 :=
 \left|\cO_{\Pj^N}(d)\right|
\]
such that every hypersurface \(H\in\mathcal H_d^\circ\) has the
following properties.
\begin{enumerate}[label=\textnormal{(\arabic*)}]
\item
Every geometric fibre of
\[
 U\cap H\longrightarrow V
\]
is nonempty and pure of dimension \(r-1\).

\item
We have
\[
 \dim(B\cap H)<r-1.
\]
If \(\dim B=0\), our choice of \(H\) makes this intersection empty.

\item
For every \(i\), the intersection \(U_i\cap H\) is empty or is a
smooth effective Cartier divisor on \(U_i\).  In particular, it has
pure codimension one in \(U_i\).
\end{enumerate}
Here all intersections are scheme-theoretic, and we use the
convention \(\dim\varnothing=-\infty\).
\end{lemma}

\begin{proof}
We impose the three conditions separately.

First consider the fibres of \(f\).  For all sufficiently large
\(d\), there is a dense open subset
\[
 \mathcal H_{d,\mathrm{fib}}\subseteq\mathcal H_d
\]
such that no \(H\in\mathcal H_{d,\mathrm{fib}}\) contains an
irreducible component of any geometric fibre of \(f\).  Let
\(E\) be such a component.  It is projective and has positive
dimension.  Hence a nonzero section of the ample line bundle
\(\cO_E(d)\) must have a zero.  It follows that \(E\cap H\) is
nonempty and has pure codimension one in \(E\).  Therefore, for any
\(H\in\mathcal H_{d,\mathrm{fib}}\),
\[
 (U\cap H)_{\bar v}
 =
 U_{\bar v}\cap H
\]
is nonempty and pure of dimension \(r-1\)
for every geometric point \(\bar v\to V\).
This is the content of \cite[Lemma~3.1]{KreschVistoli}.

Next consider \(B\).  Let \(B_1,\ldots,B_m\) be its irreducible
components.  For each \(B_j\), the hypersurfaces containing \(B_j\)
form a proper closed linear subspace of \(\mathcal H_d\).  We remove
these finitely many subspaces.  This gives a dense open subset
\[
 \mathcal H_{d,B}\subseteq\mathcal H_d.
\]
For every \(H\in\mathcal H_{d,B}\), no irreducible component of
\(B\) is contained in \(H\).  Therefore
\[
 \dim(B\cap H)\leq\dim B-1<r-1.
\]
This proves~(2).  

Finally, fix \(i\).  The degree-\(d\) hypersurfaces restrict to a
base-point-free linear system on \(U_i\).  By Bertini's theorem
\cite[Chapter~III, Corollary~10.9]{Hartshorne}, there is a dense open
subset
\[
 \mathcal H_{d,i}\subseteq\mathcal H_d
\]
such that \(H\cap U_i\) is smooth.  We may also require that \(H\)
contain no irreducible component of \(U_i\).  Thus \(H\cap U_i\) is
empty or is a smooth effective Cartier divisor on \(U_i\).

Now take
\[
 \mathcal H_d^\circ
 :=
 \mathcal H_{d,\mathrm{fib}}
 \cap
 \mathcal H_{d,B}
 \cap
 \bigcap_{i=1}^s\mathcal H_{d,i}.
\]
This is a finite intersection of dense open subsets of the
irreducible space \(\mathcal H_d\).  It is therefore dense and open.
\end{proof}

\begin{proposition}[Stratified Kresch--Vistoli cover]
\label{prop:stratified-kv}
Let \(\mathscr X\) be a smooth separated quotient
Deligne--Mumford stack over \(\C\).  Assume that its coarse moduli
space is quasi-projective.  Let
\(\{\mathscr Z_\nu\}_{\nu\in N}\) be a finite family of smooth closed
substacks.  Then there are a smooth quasi-projective scheme \(Y\) and
a finite flat surjective lci map
\[
 q:Y\longrightarrow\mathscr X
\]
such that, for every \(\nu\in N\), the scheme
\[
 Y\times_{\mathscr X}\mathscr Z_\nu
\]
is smooth over \(\C\).
\end{proposition}

\begin{proof}
The proof of \cite[Theorem~2.1]{KreschVistoli} gives a smooth
projective morphism
\[
 \pi:\mathscr P\longrightarrow\mathscr X
\]
of relative dimension \(r>0\).  It also gives an open substack
\(Q\subseteq\mathscr P\) with the following properties.  The stack
\(Q\) is a scheme.  The map \(Q\to\mathscr X\) is smooth and
surjective.  If
\[
 \mathscr S:=\mathscr P\setminus Q,
\]
with its reduced structure, then
\[
 \dim\mathscr S<r.
\]

Let \(U\) and \(V\) be the coarse spaces of \(\mathscr P\) and
\(\mathscr X\).  Kresch and Vistoli show that \(U\) is a
quasi-projective scheme.  The induced map
\[
 f:U\longrightarrow V
\]
is proper and surjective.  Its geometric fibres are nonempty and
pure of dimension \(r\).  Indeed, these fibres are finite quotients
of the geometric fibres of \(\mathscr P\to\mathscr X\).  Let
\(B\subseteq U\) be the image of \(\mathscr S\).  The coarse map
identifies \(Q\) with \(U\setminus B\), and
\[
 \dim B<r.
\]

For every \(\nu\), put
\[
 Q_\nu:=Q\times_{\mathscr X}\mathscr Z_\nu.
\]
This is a closed subscheme of \(Q\).  It is smooth over \(\C\).
Indeed, \(Q_\nu\to\mathscr Z_\nu\) is the base change of the smooth
map \(Q\to\mathscr X\), and \(\mathscr Z_\nu\) is smooth over
\(\C\).  Thus \(Q\) and all the schemes \(Q_\nu\) are smooth
locally closed subschemes of \(U\).

Fix a locally closed immersion \(U\hookrightarrow\Pj^N\).  We now
apply \cref{lem:relative-bertini} repeatedly.  Set
\[
 U^{(0)}=U,\qquad B^{(0)}=B,\qquad
 Q^{(0)}=Q,\qquad Q_\nu^{(0)}=Q_\nu.
\]
Suppose that \(H_1,\ldots,H_{j-1}\) have been chosen.  Apply the
lemma to
\[
 U^{(j-1)}\longrightarrow V,
 \qquad B^{(j-1)}\subseteq U^{(j-1)},
\]
and to the finite family
\[
 Q^{(j-1)},\qquad \{Q_\nu^{(j-1)}\}_{\nu\in N}.
\]
Choose a sufficiently large degree \(d_j\), and choose a
hypersurface \(H_j\) in the dense open set given by the lemma.  Put
\begin{align*}
 U^{(j)}
 &:=U^{(j-1)}\cap H_j,\\
 B^{(j)}
 &:=B^{(j-1)}\cap H_j,\\
 Q^{(j)}
 &:=Q^{(j-1)}\cap H_j,\\
 Q_\nu^{(j)}
 &:=Q_\nu^{(j-1)}\cap H_j.
\end{align*}
The lemma gives the following facts:
\begin{enumerate}[label=\textnormal{(\arabic*)}]
\item every geometric fibre of \(U^{(j)}\to V\) is nonempty and
pure of dimension \(r-j\);
\item \(\dim B^{(j)}<r-j\);
\item \(Q^{(j)}\) and every \(Q_\nu^{(j)}\) are empty or smooth.
Moreover, each nonempty intersection in~(3) is an effective Cartier
divisor in the preceding one.
\end{enumerate}
These facts also verify the hypotheses needed for the next step.
Thus the construction continues until \(j=r\).

Now \(\dim B^{(r)}<0\), so \(B^{(r)}=\varnothing\).  Hence
\[
 Y:=U^{(r)}=Q^{(r)}\subseteq Q.
\]
The scheme \(Y\) is smooth and quasi-projective.  The map
\(Y\to V\) is proper and has nonempty zero-dimensional geometric
fibres.  It is finite by \cite[Tag~02LS]{StacksProject}.  It is
surjective because every geometric fibre is nonempty.

Let
\[
 q:Y\longrightarrow\mathscr X
\]
be the map induced by \(Y\subseteq Q\).  The construction of Kresch
and Vistoli identifies \(Q\) with \(U\setminus B\).  Hence the
inverse image of the closed subscheme \(Y\subseteq U\) under
\(\mathscr P\to U\) is again \(Y\).  Thus \(Y\) is closed in
\(\mathscr P\), so \(q\) is proper.  The map \(q\) is representable
because \(\mathscr X\) is Deligne--Mumford.  Its geometric fibres are
zero-dimensional, since they map into the fibres of \(Y\to V\).
Hence \(q\) is quasi-finite.  Since it is proper, it is finite by
\cite[Tag~02LS]{StacksProject}.  The coarse moduli
map identifies the geometric points of \(\mathscr X\) with those of
\(V\).  Since \(Y\to V\) is surjective, \(q\) is also surjective.

At each step, \(Q^{(j)}\) is an effective Cartier divisor in
\(Q^{(j-1)}\).  Therefore \(Y\to Q\) is a regular immersion.  Since
\(Q\to\mathscr X\) is smooth, the map \(q\) is lci.

It remains to prove flatness.  Take a smooth scheme chart
\(T\to\mathscr X\).  The scheme \(T\) is regular.  Moreover,
\(Y_T:=Y\times_{\mathscr X}T\) is smooth over \(Y\), and hence it is
smooth over \(\C\).  In particular, \(Y_T\) is Cohen--Macaulay.  The
map \(Y_T\to T\) is finite.  Put
\(Q_T:=Q\times_{\mathscr X}T\).  The map \(Q_T\to T\) is smooth of
relative dimension \(r\), while \(Y_T\to Q_T\) is a regular
immersion of codimension \(r\).  Let \(y\in Y_T\), and let \(t\) be
its image in \(T\).  Since \(Y_T\to T\) is finite, \(y\) is a closed
point of the fibre \((Q_T)_t\).  Hence
\[
 \dim\mathcal O_{Q_T,y}=\dim\mathcal O_{T,t}+r,
 \qquad
 \dim\mathcal O_{Y_T,y}
 =\dim\mathcal O_{Q_T,y}-r
 =\dim\mathcal O_{T,t}.
\]
Miracle flatness
\cite[Tag~00R4]{StacksProject} shows that \(Y_T\to T\) is flat.
Flatness is smooth-local on the target
\cite[Tag~06F7]{StacksProject}.  Thus \(q\) is flat.

Finally, scheme-theoretic base change gives
\[
 Y\times_{\mathscr X}\mathscr Z_\nu=Q_\nu^{(r)}.
\]
The right-hand side is smooth.  This is the required property.
\end{proof}

\begin{theorem}[Projective cover with SNC boundary]
\label{thm:projective-adapted-cover}
Let
\((\mathscr X,\boldsymbol{\mathscr D},\bar w)\) be a projective good
stack compactification.  There is a finite flat surjective map
\[
 q:Y\longrightarrow\mathscr X
\]
from a smooth projective scheme such that
\((q^{-1}\mathscr D)_{\red}\) is SNC.

Put
\[
 g:=\bar w\circ q,
 \qquad
 P_Y:=g^*(\infty)=q^*\mathscr P.
\]
Then \((Y,(q^{-1}\mathscr D)_{\red},g)\) satisfies the hypotheses of
\cref{thm:esy}.
\end{theorem}

\begin{proof}
We may work one connected component at a time.  The stack
\(\mathscr X\) is a quotient stack by
\cite[Theorem~4.4]{KreschGeometry}.  Its coarse space is projective.

Apply \cref{prop:stratified-kv} to all nonempty strata
\(\mathscr D_J\).  We also include
\(\mathscr D_\varnothing=\mathscr X\).  We obtain a smooth
quasi-projective scheme \(Y\) and a finite flat surjective lci map
\(q:Y\to\mathscr X\) such that
\[
 Y\times_{\mathscr X}\mathscr D_J
\]
is smooth for every stratum in this family.  We now check the
boundary.  Put
\[
 D_i:=Y\times_{\mathscr X}\mathscr D_i.
\]
Flat pullback preserves effective Cartier divisors.  Thus every
nonempty \(D_i\) is an effective Cartier divisor on \(Y\).  It is
smooth by the property above.  Hence it is reduced.

For every \(J\subseteq I\), scheme-theoretic base change gives
\[
 \bigcap_{j\in J}D_j
 =Y\times_{\mathscr X}\mathscr D_J.
\]
This scheme is smooth by the property above.  Moreover, flat base
change preserves the codimension of the regular immersion
\(\mathscr D_J\hookrightarrow\mathscr X\).  Thus every nonempty
intersection has pure codimension \(|J|\) in \(Y\).  Therefore
\(\{D_i\}_{i\in I}\) is a labelled SNC divisor.  The support of
\(q^{-1}\mathscr D\) is the union of the supports of the \(D_i\).
Since every nonempty \(D_i\) is reduced, the reduced divisor with
this support is
\[
 \bigcup_{i\in I}D_i=(q^{-1}\mathscr D)_{\red}.
\]
This proves that \((q^{-1}\mathscr D)_{\red}\) is SNC.

The map \(q\) is finite, hence proper.  The stack \(\mathscr X\) is
proper.  Hence \(Y\) is proper.  A proper quasi-projective scheme over
\(\C\) is projective.
\end{proof}
 \section{Pullback and trace on the lattices}
\label{sec:trace}

Let \(\mathscr X\) be smooth and connected, 
\(\boldsymbol{\mathscr D}\) be a labelled SNC divisor on
\(\mathscr X\),
\[
 \mathscr P=\sum_i m_i\mathscr D_i
\]
be effective and supported on \(\mathscr D\).  

Let
\[
 q:Y\longrightarrow\mathscr X
\]
be finite flat and surjective.  
Assume that \(Y\) is smooth.  Let
\(\deg(q)\) denote the rank of \(q_*\cO_Y\).
Assume that
\[
 D_Y:=(q^{-1}\mathscr D)_{\red}
\]
is SNC.  
Put \(P_Y=q^*\mathscr P\).

The aim of this section is to construct a filtered retraction of
pullback on the Kontsevich complexes.  We first work with the rounded
logarithmic lattices.  Pullback preserves these lattices by
\cref{prop:weighted-pullback}.  We then define rational trace.  It
satisfies
\[
 \Tr_q^a\circ q^*=\deg(q)\,\id.
\]
Thus it remains to prove that rational trace also preserves the
rounded logarithmic lattices.  This is the content of
\cref{thm:weighted-trace}.  After that theorem,
\(\deg(q)^{-1}\Tr_q\) gives
the required filtered retraction.  The passage from the lattices to
the Kontsevich complexes is formal.  It uses only the projection
formula and the compatibility of trace with the de Rham differential.

\subsection{Pullback on the lattices}

\begin{proposition}[Weighted logarithmic pullback]
\label{prop:weighted-pullback}
For every \(a\) and every \(c\in\Q\), pullback restricts to an
\(\cO_{\mathscr X}\)-linear map
\begin{equation}
 q^*:
 \mathcal L_{\mathscr X}^a(c)
 \longrightarrow
 q_*\mathcal L_Y^a(c).
 \label{eq:weighted-pullback}
\end{equation}
\end{proposition}

\begin{proof}
Since
\[
 D_Y=(q^{-1}\mathscr D)_{\red},
\]
pullback gives a natural map
\[
 q^*\Omega_{\mathscr X}^a(\log\mathscr D)
 \longrightarrow
 \Omega_Y^a(\log D_Y).
\]
It remains to compare the divisor twists.  We claim that
\[
 q^*\lfloor c\mathscr P\rfloor
 \leq
 \lfloor cP_Y\rfloor.
\]
Let \(T\) be a prime divisor of \(Y\) over \(\mathscr D_i\), and let
\(e_T\) be its ramification index.  The coefficients of these two
divisors along \(T\) are
\[
 e_T\lfloor cm_i\rfloor
 \qquad\text{and}\qquad
 \lfloor ce_Tm_i\rfloor.
\]
Since
\[
 e_T\lfloor cm_i\rfloor
 \leq
 \lfloor ce_Tm_i\rfloor,
\]
the claimed inequality follows.  Hence
\[
 \mathcal O_Y\bigl(q^*\lfloor c\mathscr P\rfloor\bigr)
 \longrightarrow
 \mathcal O_Y\bigl(\lfloor cP_Y\rfloor\bigr).
\]
Tensoring this inclusion with logarithmic pullback gives
\(q^*\mathcal L_{\mathscr X}^a(c)\to\mathcal L_Y^a(c)\) on
\(Y\).  The adjoint map is \eqref{eq:weighted-pullback}.
\end{proof}

\subsection{Trace on rational forms}

Let \(\mathscr K_{\mathscr X}\) be the sheaf of rational functions.

\begin{lemma}[Rational trace]
\label{lem:rational-trace}
For every \(a\geq0\), field trace defines a map on rational forms
\[
 \Tr_q^a:
 q_*\Omega_Y^a\otimes\mathscr K_{\mathscr X}
 \longrightarrow
 \Omega_{\mathscr X}^a\otimes\mathscr K_{\mathscr X}.
\]
This map preserves regular forms.  Thus it has a canonical
\(\cO_{\mathscr X}\)-linear restriction
\[
 \Tr_q^a:q_*\Omega_Y^a\longrightarrow\Omega_{\mathscr X}^a.
\]
We use the same symbol for the regular map and its rationalization.

The regular map and its rationalization commute with
\(\acute{e}\)tale base change.  If \(\omega\) is a rational
\(r\)-form on \(\mathscr X\) and \(\eta\) is a rational \(s\)-form on
\(Y\), then
\begin{equation}
 \Tr_q^{r+s}(q^*\omega\wedge\eta)
 =\omega\wedge\Tr_q^s(\eta).
 \label{eq:rational-trace-linearity}
\end{equation}
In particular, on rational \(a\)-forms,
\begin{equation}
 \Tr_q^a\circ q^*=\deg(q)\,\id.
 \label{eq:rational-pull-trace}
\end{equation}
For every rational \(a\)-form \(\eta\) on \(Y\), the trace also
commutes with the de Rham differential:
\begin{equation}
 \Tr_q^{a+1}(d\eta)=d\Tr_q^a(\eta).
 \label{eq:trace-d}
\end{equation}
\end{lemma}

\begin{proof}
Work on a connected \(\acute{e}\)tale scheme chart
\(X\to\mathscr X\).  Put
\[
 q_X:Y_X:=Y\times_{\mathscr X}X\longrightarrow X.
\]
Let \(K=K(X)\).  If \(L_i\) is the function field of a connected
component of \(Y_X\), then \(L_i/K\) is finite and separable.  Hence
\[
 \Omega_{L_i/\C}^a
 \simeq L_i\otimes_K\Omega_{K/\C}^a.
\]
The rational trace is defined by
\[
 \Tr_{L_i/K}^a(b\otimes\omega)
 =\Tr_{L_i/K}(b)\omega,
\]
and by summing over \(i\).

The map \(q_X\) is finite syntomic.  Indeed, it is finite flat and of
finite presentation.  It is also a local complete intersection
morphism.  Factor it as its graph
\[
 Y_X\longrightarrow Y_X\times_{\C}X
\]
followed by the smooth projection to \(X\).  The graph is a regular
immersion because it is obtained from the diagonal of the smooth
\(\C\)-scheme \(X\) by base change.  The finite syntomic de Rham
trace therefore gives a map of complexes
\[
 \Tr_{q_X}:
 (q_X)_*\Omega_{Y_X/\C}^{\bullet}
 \longrightarrow
 \Omega_{X/\C}^{\bullet}.
\]
It commutes with base change by \cite[Tag~0FLB]{StacksProject}.  It is
a map of de Rham complexes and satisfies the projection formula by
\cite[Tag~0FLC]{StacksProject}.

Over the dense open set where \(q_X\) is finite
\(\acute{e}\)tale, this trace is given by
\[
 \Tr_{q_X}^a(b\,q_X^*\omega)
 =\Tr_{Y_X/X}(b)\omega.
\]
Its rationalization is therefore the field trace defined above.  In
particular, rational trace preserves regular forms.  Base-change
compatibility gives descent from the \(\acute{e}\)tale charts to
\(\mathscr X\).  The projection formula gives
\eqref{eq:rational-trace-linearity}.  Since
\(\Tr_q(1)=\deg(q)\), it also
gives \eqref{eq:rational-pull-trace}.  Since the trace is a map of de
Rham complexes, it gives \eqref{eq:trace-d}.
\end{proof}

By \cref{prop:weighted-pullback}, pullback already gives
\[
 q^*:\mathcal L_{\mathscr X}^a(c)
 \longrightarrow q_*\mathcal L_Y^a(c).
\]
By \eqref{eq:rational-pull-trace}, rational trace is \(\deg(q)\)
times a left inverse of this map.  Therefore the desired retraction on the
lattices will follow once we prove that rational trace restricts to
\[
 q_*\mathcal L_Y^a(c)
 \longrightarrow \mathcal L_{\mathscr X}^a(c).
\]
Thus the problem is reduced to \cref{thm:weighted-trace}.  We first
prove the local DVR statement needed for its proof.

\subsection{The local DVR calculation}
\label{subsec:local-dvr-trace}

The weighted trace theorem is a codimension-one statement.  We first
prove the local result that will be used at a prime divisor.

Let \(A\) be a DVR which is essentially of finite type over \(\C\).
Let \(x\) be a uniformizer.  Let \(C\) be a finite flat regular
\(A\)-algebra.  The ring \(C\) is a finite product of regular
semilocal rings of dimension one.  For every maximal ideal \(T\) of
\(C\), the local ring \(C_T\) is a DVR.  Let \(e_T\) be its
ramification index over \(A\).

Let \(A^{\mathrm{sh}}\) be a strict henselization of \(A\).
By \cite[Propositions~2.8.18 and~2.8.20]{FuEtaleCohomology}, there is
a finite product decomposition
\begin{equation}
 C\otimes_A A^{\mathrm{sh}}
 \simeq
 \prod_\lambda B_\lambda,
 \label{eq:sh-decomposition}
\end{equation}
where every \(B_\lambda\) is a strictly henselian DVR.  Each factor
lies over a unique maximal ideal \(T(\lambda)\) of \(C\).  The local
map
\[
 C_{T(\lambda)}\longrightarrow B_\lambda
\]
is ind-\(\acute{e}\)tale.  Thus its ramification index over
\(A^{\mathrm{sh}}\) is \(e_{T(\lambda)}\).

\begin{lemma}[Strict-henselian Kummer structure]
\label{lem:strict-henselian-kummer-structure}
For each \(\lambda\), put \(e_\lambda=e_{T(\lambda)}\).  There is an
element \(t_\lambda\in B_\lambda\) and an
\(A^{\mathrm{sh}}\)-isomorphism
\[
 B_\lambda
 \simeq
 A^{\mathrm{sh}}[T]/(T^{e_\lambda}-x),
 \qquad
 T\longmapsto t_\lambda.
\]
In particular, \(t_\lambda\) is a uniformizer and
\(t_\lambda^{e_\lambda}=x\).

Note that, if \(\bar t_\lambda\) is the image of a uniformizer of
\(C_{T(\lambda)}\), then \(t_\lambda\) and \(\bar t_\lambda\) differ
by a unit.  Hence
\[
 \Omega_{B_\lambda}^a(\log t_\lambda)
 =\Omega_{B_\lambda}^a(\log\bar t_\lambda),
 \qquad
 (t_\lambda)=(\bar t_\lambda).
\]
\end{lemma}

\begin{proof}
Let \(K^{\mathrm{sh}}\) and \(L_\lambda\) be the fraction fields of
\(A^{\mathrm{sh}}\) and \(B_\lambda\).  The residue field extension
is trivial.  Indeed, it is finite and separable, while the residue
field of \(A^{\mathrm{sh}}\) is separably closed.  Thus
\[
 [L_\lambda:K^{\mathrm{sh}}]=e_\lambda
\]
by \cite[Proposition~8.1.1(i)]{FuEtaleCohomology}.  The extension is
totally ramified.  It is tame because the residue characteristic is
zero.  The Kummer classification gives
\[
 L_\lambda
 \simeq
 K^{\mathrm{sh}}[T]/(T^{e_\lambda}-x);
\]
see \cite[Proposition~8.1.3(i)]{FuEtaleCohomology}.  The polynomial
\(T^{e_\lambda}-x\) is Eisenstein.  Its integral model is the integral
closure of \(A^{\mathrm{sh}}\) in \(L_\lambda\).  This integral
closure is \(B_\lambda\).  This proves the first assertion.

The image \(\bar t_\lambda\) is a uniformizer because
\(C_{T(\lambda)}\to B_\lambda\) is ind-\(\acute{e}\)tale.  Any two
uniformizers of a DVR differ by a unit.  If
\(t_\lambda=w\bar t_\lambda\), then
\[
 d\log t_\lambda=d\log\bar t_\lambda+d\log w,
 \qquad
 d\log w\in\Omega_{B_\lambda}^1.
\]
This proves the last two equalities.
\end{proof}

\begin{lemma}[Universal Kummer calculation]
\label{lem:universal-kummer-calculation}
Put
\[
 R=\C[u]_{(u)},
 \qquad
 S=\C[v]_{(v)},
\]
and define \(R\to S\) by \(u\mapsto v^e\).  Then the logarithmic
pullback map is an isomorphism
\begin{equation}
 S\otimes_R\Omega_R^1(\log u)
 \simeq
 \Omega_S^1(\log v).
 \label{eq:kummer-log-etale}
\end{equation}
Let \(F=\C(u)\) and \(E=\C(v)\).  For every \(k\in\Z\),
\begin{equation}
 \Tr_{E/F}(v^{-k}S)
 =u^{-\lfloor k/e\rfloor}R.
 \label{eq:universal-kummer-ideal}
\end{equation}
\end{lemma}

\begin{proof}
The identity
\[
 d\log u=e\,d\log v
\]
gives \eqref{eq:kummer-log-etale}, since \(e\) is invertible in
\(\C\).

Let \(\zeta\in\C\) be a primitive \(e\)-th root of unity.  The
extension \(E/F\) is Galois, with conjugates
\(v,\zeta v,\ldots,\zeta^{e-1}v\).  Hence
\[
 \Tr_{E/F}(v^j)
 =
 \begin{cases}
  0,&e\nmid j,\\
  e u^{j/e},&e\mid j.
 \end{cases}
\]
The fractional ideal \(v^{-k}S\) is free over \(R\), with basis
\[
 v^{-k},v^{1-k},\ldots,v^{e-1-k}.
\]
Among these exponents, the unique multiple of \(e\) is
\(-e\lfloor k/e\rfloor\).  Its trace is
\(e u^{-\lfloor k/e\rfloor}\).  Since \(e\in\C^\times\), this proves
\eqref{eq:universal-kummer-ideal}.
\end{proof}

Let \(K\) be the fraction field of \(A\).  Since \(C\) is finite and
regular, its generic algebra is a product of finite separable field
extensions:
\[
 C_K:=C\otimes_AK=\prod_j L_j.
\]
We write
\[
 \Tr_{C_K/K}^a:=\sum_j\Tr_{L_j/K}^a.
\]
Let
\[
 D_C:=\sum_T T, \qquad  P_C:=m\sum_Te_TT,
\]
where \(T\) runs over the maximal ideals of \(C\), \(m\geq0\).

\begin{lemma}[Semilocal DVR trace]
\label{lem:semilocal-dvr-trace}
For every \(a\geq0\), \(m\geq0\), and \(c\in\Q\),
\begin{equation}
 \Tr_{C_K/K}^a\!\left(
  \Omega_C^a(\log D_C)(\lfloor cP_C\rfloor)
 \right)
 \subseteq
 \Omega_A^a(\log(x))(\lfloor cm\rfloor(x)).
 \label{eq:semilocal-dvr-log}
\end{equation}
\end{lemma}

\begin{proof}
We check the containment after the faithfully flat ind-\(\acute{e}\)tale
base change \(A\to A^{\mathrm{sh}}\).  By
\eqref{eq:sh-decomposition}, the source then splits into the lattices
\[
 t_\lambda^{-k_\lambda}
 \Omega_{B_\lambda}^a(\log t_\lambda),
 \qquad
 k_\lambda:=\lfloor ce_\lambda m\rfloor.
\]
Here we used \cref{lem:strict-henselian-kummer-structure} to replace
the image of a local uniformizer by \(t_\lambda\).

Fix one factor.  The presentation
\[
 B_\lambda
 \simeq
 A^{\mathrm{sh}}[t_\lambda]/(t_\lambda^{e_\lambda}-x)
\]
and the relation
\(d\log x=e_\lambda d\log t_\lambda\) give
\begin{equation}
 \Omega_{B_\lambda}^a(\log t_\lambda)
 \simeq
 B_\lambda\otimes_{A^{\mathrm{sh}}}
 \Omega_{A^{\mathrm{sh}}}^a(\log x).
 \label{eq:kummer-log-base-change}
\end{equation}
This is also the base change of the Kummer logarithmic calculation
\eqref{eq:kummer-log-etale}.

Take \(e=e_\lambda\) in \cref{lem:universal-kummer-calculation}.  The
map
\[
 R\longrightarrow A^{\mathrm{sh}},
 \qquad
 u\longmapsto x,
\]
is flat.  The Kummer isomorphism identifies
\(B_\lambda\) with \(A^{\mathrm{sh}}\otimes_RS\), with
\(v\mapsto t_\lambda\).  Finite locally free trace commutes with base
change.  Thus \eqref{eq:universal-kummer-ideal} gives
\begin{equation}
 \Tr_{L_\lambda/K^{\mathrm{sh}}}
 \left(t_\lambda^{-k_\lambda}B_\lambda\right)
 =x^{-\lfloor k_\lambda/e_\lambda\rfloor}A^{\mathrm{sh}}.
 \label{eq:kummer-trace-after-base-change}
\end{equation}
By the projection formula \eqref{eq:rational-trace-linearity},
\begin{align*}
 &\Tr_{L_\lambda/K^{\mathrm{sh}}}^a\!\left(
  t_\lambda^{-k_\lambda}
  \Omega_{B_\lambda}^a(\log t_\lambda)
 \right)
 \subseteq
 x^{-\lfloor k_\lambda/e_\lambda\rfloor}
 \Omega_{A^{\mathrm{sh}}}^a(\log x)
 =
 x^{-\lfloor cm\rfloor}
 \Omega_{A^{\mathrm{sh}}}^a(\log x).
\end{align*}
For the last equality, we used
\[
 \left\lfloor\frac{\lfloor eu\rfloor}{e}\right\rfloor
 =\lfloor u\rfloor
 \qquad(e\geq1).
\]

After base change, the total trace is the sum of the traces on the
factors \(B_\lambda\).  The same containment therefore holds for the
total trace.  Faithful flatness descends it to \(A\).  This proves
\eqref{eq:semilocal-dvr-log}.
\end{proof}

\subsection{The weighted logarithmic trace}

\begin{theorem}[Weighted logarithmic trace]
\label{thm:weighted-trace}
For every \(a\geq0\) and every \(c\in\Q\), rational trace restricts
to an \(\cO_{\mathscr X}\)-linear map
\begin{equation}
 \Tr_{q,c}^a:
 q_*\mathcal L_Y^a(c)
 \longrightarrow
 \mathcal L_{\mathscr X}^a(c).
 \label{eq:weighted-trace}
\end{equation}
\end{theorem}

\begin{proof}
The question is \(\acute{e}\)tale local on \(\mathscr X\).  Choose a
connected \(\acute{e}\)tale scheme chart \(X\to\mathscr X\).  Put
\[
 q_X:Y_X:=Y\times_{\mathscr X}X\longrightarrow X.
\]
Since \(X\) is regular, it is enough to check the required
containment at every codimension-one point of \(X\).  We use the same
notation for the pullbacks of \(\mathscr D\) and \(\mathscr P\) to
\(X\).

Let \(H\subset X\) be a prime divisor.  Put
\[
 A=\cO_{X,\eta_H},
 \qquad
 K=K(X).
\]
The base change of \(q\) to \(\operatorname{Spec} A\) has the form
\[
 \operatorname{Spec} C
 :=Y_X\times_X\operatorname{Spec} A
 \longrightarrow\operatorname{Spec} A.
\]
The algebra \(C\) is finite flat over \(A\).  It is regular because
it is a localization of the regular scheme \(Y_X\).  Thus \(C\) is
exactly of the form used in \cref{lem:semilocal-dvr-trace}.  The
localization of rational trace at \(\eta_H\) is the total trace from
the generic algebra \(C_K\) to \(K\).

Suppose first that \(H\) is a component of the boundary.  Let \(x\)
be a uniformizer of \(A\), and let \(m\) be the multiplicity of
\(\mathscr P\) along \(H\).  The maximal ideals \(T\) of \(C\)
correspond to the prime divisors of \(Y_X\) over \(H\).  If \(e_T\)
is the ramification index, then the multiplicity of \(P_Y\) along
\(T\) is \(e_Tm\).  Over \(\operatorname{Spec} C\), the reduced boundary and the
polar divisor are therefore
\[
 D_C=\sum_TT,
 \qquad
 P_C=m\sum_Te_TT.
\]
The required containment at \(\eta_H\) is now precisely
\eqref{eq:semilocal-dvr-log}.

Suppose next that \(H\) is not a boundary component.  Then the source
lattice is the regular differential lattice at every point over
\(\eta_H\).  The regularity statement in
\cref{lem:rational-trace} gives
\[
 \Tr_{C_K/K}^a(\Omega_C^a)\subseteq\Omega_A^a.
\]
This also covers ramification over \(H\).

Thus rational trace lies in \(\mathcal L_{\mathscr X}^a(c)\) at every
codimension-one point of \(X\).  Since this lattice is locally free,
the containment holds on \(X\).  The construction is compatible with
the \(\acute{e}\)tale groupoid.  It therefore descends to
\(\mathscr X\) and gives \eqref{eq:weighted-trace}.
\end{proof}

\begin{proposition}[Filtered pullback and trace]
\label{prop:kontsevich-pull-trace}
Let \((\mathscr X,\mathscr D,\bar w)\) be a good compactification,
and put \(g=\bar w\circ q\).  For every
\(\alpha\in\Q\cap[0,1)\), pullback and trace give filtered chain maps
\begin{align*}
 q^*&:
 K_{\mathscr X,\bar w}(\alpha)
 \longrightarrow q_*K_{Y,g}(\alpha),\\
 \Tr_q&:
 q_*K_{Y,g}(\alpha)
 \longrightarrow K_{\mathscr X,\bar w}(\alpha).
\end{align*}
They satisfy
\begin{equation}
 \Tr_q\circ q^*
 =\deg(q)\,\id_{K_{\mathscr X,\bar w}(\alpha)}.
 \label{eq:kontsevich-pull-trace}
\end{equation}
Hence \(\deg(q)^{-1}\Tr_q\) is a filtered retraction of \(q^*\), and
\(K_{\mathscr X,\bar w}(\alpha)\) is a filtered direct summand of
\(q_*K_{Y,g}(\alpha)\).
\end{proposition}

\begin{proof}
By \cref{prop:weighted-pullback,thm:weighted-trace}, pullback and trace
preserve the rounded logarithmic lattices.  For rational forms
\(\omega\) and \(\eta\), they also satisfy
\[
 \nabla_g(q^*\omega)=q^*(\nabla_{\bar w}\omega),
 \qquad
 \nabla_{\bar w}\Tr_q(\eta)
 =\Tr_q(\nabla_g\eta).
\]
The second identity follows from \cref{lem:rational-trace}, the
projection formula, and \(dg=q^*(d\bar w)\).  Thus both maps preserve
the defining kernels of the Kontsevich lattices.  They preserve form
degree, so they preserve the stupid filtrations.  Finally,
\eqref{eq:kontsevich-pull-trace} is the restriction of
\eqref{eq:rational-pull-trace}.  
\end{proof}

\begin{corollary}[Global filtered retraction]
\label{cor:global-filtered-retraction}
For every \(\alpha\in\Q\cap[0,1)\), pullback and normalized trace
give filtered maps
\[
 R\Gamma\!\left(\mathscr X,K_{\mathscr X,\bar w}(\alpha)\right)
 \xrightarrow{\ i\ }
 R\Gamma\!\left(Y,K_{Y,g}(\alpha)\right)
 \xrightarrow{\ r\ }
 R\Gamma\!\left(\mathscr X,K_{\mathscr X,\bar w}(\alpha)\right)
\]
such that \(ri=\id\).  Thus the first filtered complex is a filtered
retract of the second one.
\end{corollary}

\begin{proof}
By \cite[Proposition~3.5(i)]{WangCompactification}, 
every term of \(K_{Y,g}(\alpha)\) is
locally free.  Since \(q\) is finite, these terms are
\(q_*\)-acyclic; see \cite[Tags~07AY and~01XB]{StacksProject}.
Termwise pushforward therefore computes the derived pushforward of
this bounded complex.  This remains true although its differential is
only \(\C\)-linear.  Hence
\[
 Rq_*K_{Y,g}(\alpha)\simeq q_*K_{Y,g}(\alpha),
 \qquad
 R\Gamma\!\left(\mathscr X,q_*K_{Y,g}(\alpha)\right)
 \simeq R\Gamma\!\left(Y,K_{Y,g}(\alpha)\right).
\]
These isomorphisms respect the finite stupid filtrations.  Now apply
derived global sections to the maps in
\cref{prop:kontsevich-pull-trace}, and divide trace by \(\deg(q)\).
\end{proof}

If \(\mathscr X\) is not connected, apply this construction to each
connected component.  Equivalently, use the locally constant rank of
\(q_*\cO_Y\) and divide by that rank on each component.
 \section{Proofs of the degeneration theorems}
\label{sec:projective}

\subsection{Finite descent}

\begin{theorem}[Finite descent of degeneration]
\label{thm:finite-descent}
Let \((\mathscr X,\mathscr D,\bar w)\) be a good stack
compactification.  Let
\[
 q:Y\longrightarrow\mathscr X
\]
be finite flat and surjective.  Assume that \(Y\) is smooth and that
\((q^{-1}\mathscr D)_{\red}\) is SNC.  Put
\(g=\bar w\circ q\).  Fix \(\alpha\in\Q\cap[0,1)\).

If the Kontsevich spectral sequence on \(Y\) degenerates at \(E_1\),
then the Kontsevich spectral sequence on \(\mathscr X\) degenerates at
\(E_1\).
\end{theorem}

\begin{proof}
It is enough to work on one connected component of \(\mathscr X\).
By \cref{cor:global-filtered-retraction}, the filtered complex on
\(\mathscr X\) is a retract of the filtered complex on \(Y\).
Functoriality gives maps on every page
\[
 E_s(\mathscr X,\alpha)
 \xrightarrow{i_s}
 E_s(Y,\alpha)
 \xrightarrow{r_s}
 E_s(\mathscr X,\alpha).
\]
They commute with \(d_s\), and \(r_si_s=\id\).  Therefore, for every
\(s\geq1\),
\begin{equation}
 d_s^{\mathscr X}
 =r_s\,d_s^Y\,i_s.
 \label{eq:pagewise-descent}
\end{equation}
By assumption, \(d_s^Y=0\) for every \(s\geq1\).  Hence
\(d_s^{\mathscr X}=0\) for every \(s\geq1\).  This proves
\(E_1\)-degeneration on \(\mathscr X\).
\end{proof}

\subsection{The two main theorems}

\begin{proof}[Proof of \Cref{thm:intro-projective}]
By \cref{thm:projective-adapted-cover}, choose a finite flat
surjective map
\[
 q:Y\longrightarrow\mathscr X
\]
such that \(Y\) is smooth and projective and
\(D_Y=(q^{-1}\mathscr D)_{\red}\) is SNC.  Put
\(g=\bar w\circ q\).  Then \(g:Y\to\Pj^1\) is a morphism and
\(g(Y\setminus D_Y)\subseteq\A^1\).  The Esnault--Sabbah--Yu theorem,
\cref{thm:esy}, gives \(E_1\)-degeneration on \(Y\).  The result now
follows from \cref{thm:finite-descent}.
\end{proof}

\begin{proof}[Proof of \Cref{thm:intro-rational}]
Fix \(\alpha\in\Q\cap[0,1)\).  By
\cref{thm:projective-anchor}, choose a projective good
compactification \(\mathscr X_0\).  Its fixed-\(\alpha\) Kontsevich
spectral sequence degenerates by \cref{thm:intro-projective}.

Let \(\mathscr X\) be the given NC rational stack compactification.
A good compactification is a special case of such a model.  Hence
\cref{thm:compactification-comparison} identifies the
fixed-\(\alpha\) spectral sequences on \(\mathscr X_0\) and
\(\mathscr X\) from \(E_1\).  The one on \(\mathscr X\) therefore
degenerates at \(E_1\).

Finally, \cref{thm:kontsevich-yu} gives \(E_1\)-degeneration for the
Yu integer slices on \(\mathscr X_0\).  The Yu part of
\cref{thm:compactification-comparison} transfers this result to
\(\mathscr X\).
\end{proof}
 
\providecommand{\bysame}{\leavevmode\hbox to3em{\hrulefill}\thinspace}
\providecommand{\MR}{\relax\ifhmode\unskip\space\fi MR }
\providecommand{\MRhref}[2]{\href{http://www.ams.org/mathscinet-getitem?mr=#1}{#2}
}
\providecommand{\href}[2]{#2}

\end{document}